\documentclass[reqno]{amsart}
\usepackage{float}
\usepackage{amsmath}
\usepackage{graphicx}
\usepackage{latexsym}
\usepackage{amsfonts}
\usepackage{amssymb}
\theoremstyle{plain}
\newtheorem{theorem}{Theorem}
\newtheorem{corollary}[theorem]{Corollary}

\newtheorem{proposition}[theorem]{Proposition}
\theoremstyle{definition}

\newtheorem{remark}[theorem]{Remark}
\newtheorem*{remark*}{Remark}

\newcommand{\pr}{\mathbf P}
\newcommand{\e}{\mathbf E}

\begin{document}
\title[Universality of local times]{Universality of local times of killed and reflected random walks.}
\author[Denisov]{Denis Denisov} 
\address{School of Mathematics, University of Manchester, Oxford Road, Manchester M13 9PL, UK}
\email{denis.denisov@manchester.ac.uk}

\author[Wachtel]{Vitali Wachtel} \address{Institut f\"ur Mathematik,
Universit\"at Augsburg, 86135 Augsburg, Germany}
\email{vitali.wachtel@mathematik.uni-augsburg.de}

\begin{abstract}
In this note we first consider local times of random walks killed
at leaving positive half-axis. We prove that the distribution of
the properly rescaled local time at point $N$ conditioned on being
positive converges towards an exponential distribution. The proof
is based on known results for conditioned random walks, which
allow to determine the asymptotic behaviour of moments of local
times. Using this information we also show that the field of local
times of a reflected random walk converges in the sense of finite
dimensional distributions. This is in the spirit of the seminal
result by Knight ~\cite{Knight63} who has shown that for the
symmetric simple random walk local times converge wealky towards
a squared Bessel process. Our result can be seen as an extension
of the second Ray-Knight theorem to all asymptotically stable random
walks.
\end{abstract}

\keywords{Random walk, Local time, second Ray-Knight theorem}
\subjclass{Primary 60G50; Secondary 60G40, 60F17} 
\maketitle
{\scriptsize
}

\section{Introduction}
Let $\{S_n\}$ be a random walk on $\mathbb{Z}$ with increments $\{X_k\}$
which are independent copies of a random variable $X$. Let $\tau^-$ be the
first weak descending ladder epoch of our random walk, that is,
$$
\tau^-:=\min\{n\geq1: S_n\leq0\}.
$$
We shall always assume that $\mathbf{E}X=0$. This implies that $\{S_n\}$
is recurrent and, in particular, $\tau^-$ is almost sure finite. Let
\begin{equation*}
\mathcal{A}:=\{1<\alpha <2;|\beta |\leq 1\}\cup \{\alpha =2,\beta =0\}
\end{equation*}%
be a subset in $\mathbb{R}^{2}.$ For $(\alpha ,\beta )\in \mathcal{A}$ and a
random variable $X$ write $X\in \mathcal{D}\left( \alpha ,\beta \right) $ if
the distribution of $X$ belongs to the domain of attraction of a stable law
with characteristic function%
\begin{equation}
G_{\alpha ,\beta }\mathbb{(}t\mathbb{)}:=\exp \left\{ -|t|^{\,\alpha
}\left( 1-i\beta \frac{t}{|t|}\tan \frac{\pi \alpha }{2}\right) \right\}
=\int_{-\infty }^{+\infty }e^{itu}g_{\alpha ,\beta }(u)du.  \label{std}
\end{equation}
This means that there exists an increasing, regularly varying with index $1/\alpha$
function $c(x)$ such that $S_n/c(n)$ converges in distribution towards the stable
law given by \eqref{std}. By $c^{-1}(x)$ we shall denote the inverse to $c(x)$
function. Clearly, $c^{-1}$ is regularly varying wih index $\alpha$.
 
Let 
$$
L(n,x) = \sum_{j=0}^n I(S_j=x),\quad x\in\mathbb{Z}
$$
denote the local time of the process $\{S_n\}$. We first consider
local times of $\{S_n\}$ killed at leaving positive half-axis.

In order to formulate our result we have to introduce some notation. Let $\tau^+$
be the first strict ascending ladder epoch and let $\chi^\pm$ denote ladder heights
corresponding to $\tau^\pm$. Define
$$
H^\pm(x):=\sum_{j=0}^\infty\pr(\chi^\pm_1+\chi^\pm_2+\ldots+\chi^\pm_j\leq x),\quad x\geq1,
$$
where $\{\chi^\pm_j\}$ are independent copies of $\chi^\pm$. Finally, let $h^\pm$
denote the mass functions of $H^\pm$, that is,
$$
h^\pm(x)=H^\pm(x)-H^\pm(x-1),\quad x\geq0.
$$
\begin{theorem}
\label{T1}
If $X\in\mathcal{D}(\alpha,\beta)$, then there exists $c_{\alpha,\beta}$ such that, 
as $N\to\infty$, for every fixed $x\geq0$,
\begin{equation}
\label{T1.2}
\pr_x\left(\frac{N}{c^{-1}(N)}L(\tau^-,N)>u\Big|L(\tau^-,N)>0\right)
\to e^{-uc_{\alpha,\beta}},\quad u>0.
\end{equation}
and
\begin{equation}
\label{T1.1}
\pr_x(L(\tau^-,N)>0)\sim c_{\alpha,\beta}U(x,N)\frac{N}{c^{-1}(N)},
\end{equation}
where
$$
U(x,N):=h^+(N-x)+\sum_{k=1}^\infty \e \left[h^+\left(N-x+\sum_{i=1}^k\chi_i^-\right);\sum_{i=1}^k\chi_i^-<x\right].
$$
\end{theorem}
The exponential distribution in \eqref{T1.2} is not surprising. Indeed, if $\{S_n\}$
hits $N$ before it leaves $(0,\infty)$ then we may assume that $\{S_n\}$ starts at $N$.
Thus the conditioned distribution of the local time is independent of the starting point.
Let $p_N$ denote the probability that $S_n$ becomes negative before it returns to $N$.
Clearly, $p_N$ is positive. Then
$$
\pr_N(L(\tau^-,N)>k)=(1-p_N)^k,\quad k\geq0.
$$
Consequently, \eqref{T1.2} is equivalent to
$$
p_N=c_{\alpha,\beta}\frac{N}{c^{-1}(N)}(1+o(1)).
$$

It is immediate from the definition of $U(x,N)$ that
$$
\pr_0(L(\tau^-,N)>0)\sim c_{\alpha,\beta}\frac{Nh^+(N)}{c^{-1}(N)}.
$$
But for positive start points $x$ one needs additional restrictions.
\begin{corollary}
\label{Cor00}
If $h^+(x)$ is long-tailed then
\begin{equation}
\label{Cor00.1}
\pr_x(L(\tau^-,N)>0)\sim c_{\alpha,\beta}H^-(x)\frac{Nh^+(N)}{c^{-1}(N)}.
\end{equation}
\end{corollary}
The assumption that $h^+$ is long-tailed is not easy to check. Furthermore,
\eqref{Cor00.1} can used only if we know the asymptotic behaviour of $h^+(N)$.
In other words, we need a strong renewal theorem for positive ladder heights.
Some  sufficient conditions for this theorem can be found in \cite{chi} and
\cite{W2012}. 

For random walks with finite second moments $h^+$ is asymptotically constant
and, moreover, one can provide an exact expression for the constant
$c_{\alpha,\beta}$.
\begin{corollary}
\label{Cor0} 
If $\sigma^2:=\e X^2<\infty$ then
\begin{equation}
\label{Cor0.2}
\pr_x\left(\frac{L(\tau^-,N)}{N}>u\Big|L(\tau^-,N)>0\right)\to e^{-u\sigma^2/2},\quad u>0.
\end{equation}
and
\begin{equation}
\label{Cor0.1}
\pr_x(L(\tau^-,N)>0)\sim \frac{\sigma^2H^-(x)}{2\e\chi^+}N^{-1}.
\end{equation}
\end{corollary}


We now turn to reflected random walks. More precisely, we shall look at local times of the
process
$$
W_{n+1}=(W_n+X_{n+1})^+,\quad n\geq0
$$
which starts at zero, that is, $W_0=0$. Set $T_0=0$ and define recursively 
$$
T_{n+1}:=\min\{k>T_n:W_k=0\}, n\geq0.
$$
We are interested in the asymptotic behaviour of local times
$$
L_W(n,x):=\sum_{j=1}^n I(W_n=x), x\geq1.
$$
Let $M(N)$ be a sequence of natural numbers. For every $N$ define a rescaled process
$$
l^{(N)}(u)=\frac{N}{c^{-1}(N)}L_W(T_{M(N)},uN),\quad u\geq0.
$$
\begin{theorem}
\label{T2}
Assume that $X\in\mathcal{D}(\alpha,\beta)$. If $h^+$ is regularly varying then there exists
a process $\mathrm{L}_{\alpha,\beta}=\{\mathrm{L}_{\alpha,\beta}(u),u\geq0\}$ such that, for
any sequence $M(N)\sim \frac{c^{-1}(N)}{Nh^+(N)}$, finite dimensional distributions of $l^{(N)}$
converge to that of $\mathrm{L}_{\alpha,\beta}$.
\end{theorem}
\begin{remark}
It follows easily from the proof of Theorem \ref{T2} that the marginals of
$\mathrm{L}_{\alpha,\beta}$ are compound Poisson distributions with exponentially distributed
jumps.\hfill$\diamond$
\end{remark}

The distribution of the limiting process is known for $\alpha=2$ only. Knight \cite{Knight63} has
shown that if $S_n$ is a simple random walk then $l^{(N)}$ converges weakly (and not only in the
sence of finite dimensional distributions) towards the square of $0$-dimensional Bessel process.
Since the limit is the same for all random walks belonging to the domain of attraction of the
normal distribution, we conclude that $\mathrm{L}_{2,0}$ is the squared $0$-dimensional Bessel
process. More precisely, $\mathrm{L}_{2,0}$ is the unique strong solution of the equation
$$
\mathrm{L}_{2,0}(t)-\mathrm{L}_{2,0}(0)=c\int_0^t\sqrt{\mathrm{L}_{2,0}(s)}dB(s),\quad c>0,
$$
where $\{B(t),t\geq0\}$ is the Brownian motion. Unfortunately, we are not able to determine the
limit for $\alpha<2$. Using Corollary 3.5 from Eisenbaum and Kaspi \cite{EK09} one can give a
characterisation of $\mathrm{L}_{\alpha,\beta}$ in terms of permanental processes. If, additionally,
the limiting stable process is symmetric, that is, $\beta=1$ then $\mathrm{L}_{\alpha,\beta}$
can be described by a squared Gaussian process.

The proof in \cite{Knight63} is based on a trick which works for simple random walks only. Let $X_n$
be Rademacher random variables, that is, $\pr(X_n=\pm1)=1/2$. Fix some $m\geq1$ and for every $n>0$
let $Q^{(m)}_n$ denote the number times $k<T_{m+1}$ such that $W_{k-1}=n$ and $W_k=n+1$. Set also
$Q^{(m)}_0=m$. The key observation in \cite{Knight63} is that $\{Q^{(m)}_n,\ n\geq0\}$ is a Markov
chain with transition kernel given by
\begin{equation}
\label{trans.kern}
p(i,j)=(-1)^j{-i \choose j}2^{-i-j},\quad i,j\geq0. 
\end{equation}
Now it is immediate that $Q^{(m)}_n$ is a martingale. The markovian structure and the martingale
property allowed Knight to prove that 
$$
q^{(N)}(u)=\frac{1}{N} Q^{(N)}(uN),\quad u>0
$$
converges weakly towards the squared Bessel process. Noting that
$L_W(T_{m+1},n)=Q^{(m)}_n+Q^{(m)}_{n-1}$, we then conclude that the limit for $l^{(N)}$ is equal to
the limit for $q^{(N)}$ multiplied by $2$. In other words, the limit for $l^{(N)}$ is again a
squared Bessel process, but with a different scaling constant $c$.

Rogers \cite{Rogers} noticed that \eqref{trans.kern} corresponds to a critical Galton-Watson
process $Z_n$ with the geometric offspring distribution and $Z_0=m$. Then one has also 
$L_W(T_{m+1},n)=Z_n+Z_{n-1}$. As a consequence, convergence of local times follows from the
corresponding results for branching processes. The idea of connecting local times and branching
processes has been recently used by Hong and Yang \cite{HY}, who have extended Knight's result to all 
left-continuous random walks with bounded jumps. This has been achieved via connecting local times
to a 2-type critical branching process. Similar to Knight's paper, this embedding into a branching
process ensures markovian and martingale properties, which help to prove weak convergence. 

Our approach is based on the derivation of asymptotics for mixed moments of local times,
which seems to go back at least to Darling and Kac \cite{DK57}. But in order to apply this
method to killed random walks one needs to know asymptotic behaviour of the corresponding
Green (renewal) function. It has become possible due to the recent result by Caravenna and
Chaumont \cite{CC13} on bridges of random walks conditioned to stay positive. This method
allows to prove convergence of finite dimensional distributions in a strightforward manner.
But at the moment we do not know how to prove the tightness of the sequence $l^{(N)}$. The 
was not a problem in papers \cite{Knight63,HY}, where the martingal structure of
Galton--Watson processes can be used.
\section{Asymptotic behaviour of moments of local times}

The following renewal theorem is crucial for our proof.
\begin{proposition}\label{prop1}
Assume that $x_N/N\to u>0$ and $y_N/N\to  v>0$. Then, there exists 
$a_{\alpha,\beta}(u,v)>0$ such that
\begin{equation}
\label{eq1}
\frac{N}{c^{-1}(N)}\sum_{n=1}^\infty\pr_{x_N}(S_n=y_N,\tau^->n) \to a_{\alpha,\beta}(u,v),
\quad N\to \infty.
\end{equation}
\end{proposition}

\begin{proof}
We split the sum in \eqref{eq1} into three parts
\begin{multline*}
\sum_{n=1}^\infty\pr_{x_N}(S_n=y_N,\tau^->n) =P_1+P_2+P_3\\
:=
\left( 
\sum_{n<\varepsilon c^{-1}(N)}+\sum_{n\in[\varepsilon c^{-1}(N),c^{-1}(N)/\varepsilon]}+\sum_{n>c^{-1}(N)/\varepsilon}
\right)\pr_{x_N}(S_n=y_N,\tau^->n).
\end{multline*}
In view of the Gnedenko local limit theorem,
$$
\pr_{x}(S_n=z) \le \frac{C}{c_n},\quad \mbox{for any } x,z\in
\mathbb Z \mbox{ and } n\ge 1.
$$
Therefore,
\begin{align}\label{p1}
  P_1\le \sum_{n<\varepsilon c^{-1}(N)}\frac{C}{c_n} \le C_1 \varepsilon^{1-1/\alpha}\frac{c^{-1}(N)}{N}.
\end{align}
Further, applying Corollary~13 from Doney \cite{Doney}, we get 
$$
\pr_{x_N} (S_n=y_N,\tau^->n) \le C \frac{H^+(x_N)H^-(y_N)}{nc_n}\le C(u,v) \frac{c^{-1}(N)}{nc_n}.
$$
This yields 
\begin{equation}\label{p3}
P_3\le C(u,v) c^{-1}(N) \sum_{n>c^{-1}(N)/\varepsilon} \frac{1}{nc_n}
\le C(u,v) \varepsilon^{1/\alpha}\frac{c^{-1}(N)}{N}.
\end{equation}
Further, applying (4.2) from Caravenna and Chaumont \cite{CC13} to every summand in the second
sum, we obtain 
\begin{align*}
  &P_2=(1+o(1)) \sum_{n\in[\varepsilon c^{-1}(N),c^{-1}(N)/\varepsilon]}
  \frac{1}{c_n} g_{\alpha,\beta}\left(\frac{y_N-x_N}{c_n}\right)
  \phi_{\alpha,\beta}\left(\frac{x_N}{c_n},\frac{y_N}{c_n}\right) \\
&=(1+o(1))\frac{c^{-1}(N)}{N} \sum_{n\in[\varepsilon c^{-1}(N),c^{-1}(N)/\varepsilon]}
  \left(\frac{n}{c^{-1}(N)}\right)^{-1/\alpha}
  \psi_{\alpha,\beta}\left(\frac{x_N}{c_n},\frac{y_N}{c_n}\right)\frac{1}{c^{-1}(N)}\\
&=(1+o(1))\frac{c^{-1}(N)}{N} \int_{\varepsilon}^{1/\varepsilon}
x^{-1/\alpha}
  \psi_{\alpha,\beta}\left(\frac{u}{x^{1/\alpha}},\frac{v}{x^{1/\alpha}}\right)dx,
\end{align*}
where 
$$
\psi_{\alpha,\beta}(a,b)=g_{\alpha,\beta}\left(b-a\right)\phi_{\alpha,\beta}(a,b)
$$
and
$$
\phi_{\alpha,\beta}(a,b) = \pr_a\left(\inf_{0\le t\le 1} Y_t\ge 0\mid Y_1=b\right),
$$
where $Y_t$ is a stable process with $Y_1$ defined by \eqref{std}. 

Combining this with \eqref{p1} and \eqref{p3} and letting $\varepsilon\to0$, we obtain
\begin{align*}
\sum_{n=1}^\infty\pr_{x_N}(S_n=y_N,\tau^->n)
\sim \frac{c^{-1}(N)}{N}\int_0^\infty
x^{-1/\alpha}
  \psi_{\alpha,\beta}\left(\frac{u}{x^{1/\alpha}},\frac{v}{x^{1/\alpha}}\right)dx.
\end{align*} 
Thus, the proof is complete.
\end{proof}
\begin{remark}
The limit $a_{\alpha,\beta}$ is the so-called $0$-potential density of a killed stable
process.\hfill$\diamond$
\end{remark}

We next give an explicit expression for $a_{2,0}$.
Using the reflection principle for the Brownian motion, one can easily obtain 
$$
g_{2,0}(a,b)=1-e^{-2ab}.
$$
Therefore,
\begin{align}a_{2,0}(u,v)= 
\frac{1}{\sqrt{2\pi}}\int_0^\infty x^{-1/2} e^{-\frac{(u-v)^2}{2x}}
\left(1-e^{-2uv/x}\right)dx.
\end{align}
Substituting $x=y^{-1}$ we obtain 
$$
\int_0^\infty x^{-1/2} e^{-\frac{(u-v)^2}{2x}}
\left(1-e^{-2uv/x}\right)dz = 
\int_0^\infty \frac{1}{y^{3/2}} \left( e^{-y\frac{(u-v)^2}{2}}-e^{-y\frac{(u+v)^2}{2}}\right)dy.
$$
According to formula 3.434(1) from \cite{GR}
$$
\int_0^\infty \frac{e^{-\nu x}-e^{-\mu x}}{x^{\rho+1}}dx = \frac{\mu^\rho-\nu^\rho}{\rho}\Gamma(1-\rho).
$$
Consequently,
\begin{equation}
\label{a-form}
a_{2,0}(u,v)=\frac{2\Gamma(1/2)}{\sqrt{2\pi}}\left[\frac{u+v}{\sqrt{2}}-\frac{|u-v|}{\sqrt{2}}\right]
=2\min\{u,v\}.
\end{equation}

We now turn to higher moments of local times.

\begin{corollary}\label{corol2}
For every $(u_0,u_1,\ldots,u_m)\in \mathbf R^{m+1}_+$,
\begin{align}
\label{eq2}
\nonumber
&\lim_{N\to\infty}\left(\frac{c^{-1}(N)}{N}\right)^{-m}
 \e_{u_0N} \left[\sum_{j_1<j_2<\ldots<j_m} I\left(S_{j_1}=u_1N,\ldots,S_{j_m}=u_mN,\tau^->j_m\right)\right]\\ 
&\hspace{2cm}=\prod_{i=0}^{m-1} a_{\alpha,\beta}(u_i,u_{i+1})
\end{align}
and
\begin{align}
\label{eq2a}
\nonumber
&\lim_{N\to\infty}\left(\frac{c^{-1}(N)}{N}\right)^{-m}
 \e_{u_0N} \left[\sum_{j_1\leq j_2\leq \ldots\leq j_m} I\left(S_{j_1}=u_1N,\ldots,S_{j_m}=u_mN,\tau^->j_m\right)\right]\\ 
&\hspace{2cm}=\prod_{i=0}^{m-1} a_{\alpha,\beta}(u_i,u_{i+1}).
\end{align}
\end{corollary}
\begin{proof}
It is immediate from the Markov property that 
\begin{align}
\label{markov}
\nonumber
&\e_{u_0N} \left[\sum_{j_1<j_2<\ldots<j_m} I\left(S_{j_1}=u_1N,\ldots,S_{j_m}=u_mN,\tau^->j_m\right)\right]\\
\nonumber
&=\sum_{j_1=1}^\infty\pr_{u_0N}(S_{j_1}=u_1N,\tau^->j_1)\sum_{j_2=j_1+1}^\infty 
\pr_{u_1N}(S_{j_2-j_1}=u_2N,\tau^->j_2-j_1)\ldots\\
&=\prod_{i=0}^{m-1} \sum_{i=1}^\infty \pr_{u_i N}(S_j=u_{i+1}N,\tau^->j).
\end{align}
Applying Proposition~\ref{prop1}, we get \eqref{eq2}. The proof of \eqref{eq2a} is identical.
\end{proof}
\begin{proposition}\label{prop3}
Let $\mathcal T_m$ be the set of permutations of $\{1,\ldots,m\}$. Then, 
\begin{align}\label{eq4}
\nonumber
&\lim_{N\to\infty}\left(\frac{c^{-1}(N)}{N}\right)^{-m} \e_{u_0N} 
\left[\prod_{i=1}^m L(\tau^-,u_iN)\right]\\ 
&\hspace{2cm}=\sum_{\sigma\in \mathcal T_m}a_{\alpha,\beta}(u_0,u_{\sigma(1)})
\prod_{i=1}^{m-1} a_{\alpha,\beta}(u_{\sigma(i)},u_{\sigma(i+1)}).
\end{align}
\end{proposition}
\begin{remark}
The right hand side of \eqref{eq4} is a specialisation of Kac's moment formula 
(see \cite{DK57}) for local times of a killed stable process.
\end{remark}

\begin{proof}
It is clear that 
\begin{align*}
 \prod_{i=1}^m L(\tau^-,u_iN) &=
\sum_{j_1,\ldots,j_m=1}^\infty
I(S_{j_1}=u_1N,\ldots,S_{j_m}=u_mN,\tau^->\max_{k\le m} j_k)\\
&\le 
\sum_{\sigma\in \mathcal T_m} 
\sum_{j_1\le j_2\le \ldots\le j_m} 
I(S_{j_1}=u_{\sigma(1)}N,\ldots,S_{j_m}=u_{\sigma(m)}N,\tau^->j_m). 
\end{align*}
Similarly,
$$
\prod_{i=1}^m L(\tau^-,u_iN) \ge 
\sum_{\sigma\in \mathcal T_m} 
\sum_{j_1< j_2\le \ldots< j_m} 
I(S_{j_1}=u_{\sigma(1)}N,\ldots,S_{j_m}=u_{\sigma(m)}N,\tau^->j_m). 
$$
Taking expectations and applying \eqref{eq2} and \eqref{eq2a}
we get the desired result. 
\end{proof}

It remains to consider the case when the random walk starts at a fixed point $x$.
Replacing $u_0N$ by $x$ in the first equality of \eqref{markov}, we get
\begin{align}
\label{markov2}
\nonumber
&\e_{x} \left[\sum_{j_1<j_2<\ldots<j_m} I\left(S_{j_1}=u_1N,\ldots,S_{j_m}=u_mN,\tau^->j_m\right)\right]\\
\nonumber
&=\sum_{j_1=1}^\infty\pr_{x}(S_{j_1}=u_1N,\tau^->j_1)\sum_{j_2=j_1+1}^\infty 
\pr_{u_1N}(S_{j_2-j_1}=u_2N,\tau^->j_2-j_1)\ldots\\
&=\sum_{j_1=1}^\infty\pr_{x}(S_{j_1}=u_1N,\tau^->j_1)
\prod_{i=1}^{m-1} \sum_{i=1}^\infty \pr_{u_i N}(S_j=u_{i+1}N,\tau^->j).
\end{align}
Thus, additionally to Proposition \ref{prop1}, we have to determine the 
asymptotic behaviour of $\sum_{j_1=1}^\infty\pr_{x}(S_{j_1}=u_1N,\tau^->j_1)$.
First, by the duality lemma for random walks, 
$$
\pr_0(S_j=l,\tau^->j) =\pr(S_j=l, j \mbox{ is a strict ascending ladder epoch}).
$$
Consequently, for each $l\geq1$,
\begin{equation}
\label{zero}
 \sum_{j=1}^\infty \pr_0(S_j=l,\tau^->j) =\sum_{k=1}^\infty \pr(\chi^+_1+\cdots+\chi^+_k=l)=h^+(l).
\end{equation}
Second, for every positive $x$ we split $\tau^-$ into descending ladder epochs.
Then, using the Markov property and \eqref{zero}, we get
$$
\sum_{j=1}^\infty \pr_x(S_j=l,\tau^->j)=U(x,N).
$$
Combining this with \eqref{markov2} and Corollary \eqref{prop1}, one can easily obtain
\begin{proposition}
\label{prop4}
As $N\to\infty$,
\begin{align}\label{eq4a}
\e_{0} \left[\prod_{i=1}^m L(\tau^-,u_iN)\right]\sim
\left(\frac{c^{-1}(N)}{N}\right)^{m-1} 
\sum_{\sigma\in \mathcal T_m}h^+(u_{\sigma(1)}N)
\prod_{i=1}^{m-1} a_{\alpha,\beta}(u_{\sigma(i)},u_{\sigma(i+1)}).
\end{align}
\end{proposition}
\section{Proof of Theorem \ref{T1}.}
Recall that the distribution of $L(\tau^-,N)$ conditioned on $L(\tau^-,N)>0$
does not depend on the starting point. Using Proposition \ref{prop3} with
$u_0=u_1=\ldots=u_m=1$, we conclude that
$$
\lim_{N\to\infty}\left(\frac{c^{-1}(N)}{N}\right)^{-m}\e_N \left[L^m(\tau^-,N)\right]
=m!\left(a_{\alpha,\beta}(1,1)\right)^m.
$$
Thus, by the method of moments,
$$
\lim_{N\to\infty}\pr_N\left(\frac{N}{c^{-1}(N)}L(\tau^-,N)>x\right)=e^{-xc_{\alpha,\beta}},
\quad x>0.
$$
Since $\pr_x(L(\tau^-,N)\geq k|L(\tau^-,N)>0)=\pr_N(L(\tau^-,N)\geq k)$ for all $k\geq1$ and
all $x$, we get \eqref{T1.2}.

As we have shown in the previous section,
$$
\e_x L(\tau^-,N)=U(x,N).
$$
Furthermore, by the Markov property,
$$
\e_x L(\tau^-,N)=\pr_x(L(\tau^-,N)>0)\e_N L(\tau^-,N).
$$
Consequently,
$$
\pr_x(L(\tau^-,N)>0)=\frac{U(x,N)}{\e_N L(\tau^-,N)}.
$$
Applying Proposition \ref{prop3} with $m=1$ and $u_0=u_1=1$ we obtain \eqref{T1.1} with
$c_{\alpha,\beta}:=1/a_{\alpha,\beta}(1,1)$.
Thus, the proof of Theorem \ref{T1} is finished.

If $\sigma^2:=\e X^2$ is finite then $c(n)=\sigma\sqrt{n}$ and $c^{-1}(N)=N^2/\sigma^2$.
Furthermore, according to  \eqref{a-form}, $c_{2,0}=1/2$. Finally, recalling that
$\e\chi^+<\infty$ for random walks with finite variance and applying the strong renewal
theorem, we get $h^+(N)\sim\frac{1}{\e\chi^+}$ for random walks with finite variance.
As a result we have the relations from Corollary \ref{Cor0}.
\section{Proof of Theorem 2}
We start with the Laplace transform of the vector
$$
\left(\frac{N}{c^{-1}(N)}L(\tau^-,u_1N),\frac{N}{c^{-1}(N)}L(\tau^-,u_2N),\ldots,
\frac{N}{c^{-1}(N)}L(\tau^-,u_mN)\right).
$$
Obviously, for every $r\geq1$,
$$
\sum_{j=0}^{2r+1}(-1)^j\frac{x^j}{j!}\leq e^{-x}\leq\sum_{j=0}^{2r}(-1)^j\frac{x^j}{j!},\ x\geq0.
$$
From this we infer that
\begin{align}
\label{up.bound}
\nonumber
&\e_0\left[\exp\left\{-\sum_{i=1}^m \lambda_i\frac{N}{c^{-1}(N)}L(\tau^-,u_iN)\right\}\right]\\
&\hspace{2cm}\leq\sum_{j=0}^{2r}\frac{(-1)^j}{j!}\left(\frac{N}{c^{-1}(N)}\right)^j
\e_0\left(\sum_{i=1}^m\lambda_iL(\tau^-,u_iN)\right)^j
\end{align}
and
\begin{align}
\label{low.bound}
\nonumber
&\e_0\left[\exp\left\{-\sum_{i=1}^m \lambda_i\frac{N}{c^{-1}(N)}L(\tau^-,u_iN)\right\}\right]\\
&\hspace{2cm}\geq\sum_{j=0}^{2r+1}\frac{(-1)^j}{j!}\left(\frac{N}{c^{-1}(N)}\right)^j
\e_0\left(\sum_{i=1}^m\lambda_iL(\tau^-,u_iN)\right)^j.
\end{align}
Let $\{\mu=(\mu_1,\mu_2,\ldots,\mu_m)\}$ be the set of $m$-dimensional multiindices.
Then, by the binomial formula,
\begin{equation}
\label{mixed.moments}
\e_0\left(\sum_{i=1}^m\lambda_iL(\tau^-,u_iN)\right)^j=
\sum_{\mu:|\mu|=j}\frac{j!}{\mu!}\lambda^\mu\e_0\left[\prod_{i=1}^mL^{\mu_i}(\tau^-,u_iN)\right]
\end{equation}
Recall that $h^+$ is assumed to be regularly varying. Thus, it follows from
Proposition~\ref{prop4} that there exists $\psi_j(u,\mu)$ such that
\begin{equation}
\label{mixed.moments2}
\e_0\left[\prod_{i=1}^mL^{\mu_i}(\tau^-,u_iN)\right]
\sim h^+(N)\left(\frac{c^{-1}(N)}{N}\right)^{j-1}\phi_j(u,\mu).
\end{equation}
Furthermore, this proposition gives also the following bound for $\phi_j$
\begin{equation}
\label{psi-bound}
\phi_j(u,\mu)\leq j! \left(\min u_j\right)^{\alpha\rho-1}
\left(\max_{k,l} a_{\alpha,\beta}(u_k,u_l)\right)^{j-1}
\end{equation}

Combining \eqref{mixed.moments} and \eqref{mixed.moments2}, we obtain
\begin{equation}
\label{mom-asymp}
\e_0\left(\sum_{i=1}^m\lambda_iL(\tau^-,u_iN)\right)^j
\sim h^+(N)\left(\frac{c^{-1}(N)}{N}\right)^{j-1} \psi_j(u,\lambda)
\end{equation}
with some function $\psi_j$ satisfying
\begin{equation}
\label{phi-bound}
\psi_j(u,\lambda)\leq j!\left(\min u_j\right)^{\alpha\rho-1}
\left(\max_{k,l} a_{\alpha,\beta}(u_k,u_l)\right)^{j-1}
\left(\sum_{k=1}^m \lambda_k \right)^j.
\end{equation}
This estimate is immediate from \eqref{mixed.moments} and \eqref{psi-bound}.

Plugging \eqref{mom-asymp} into \eqref{up.bound} and \eqref{low.bound}, we obtain
\begin{align*}
&\limsup_{N\to\infty}\frac{c^{-1}(N)}{nh^+(N)}
\left(\e_0\left[\exp\left\{-\sum_{i=1}^m \lambda_i\frac{N}{c^-(N)}L(\tau^-,u_iN)\right\}\right]-1\right)\\
&\hspace{3cm}\leq \sum_{l=1}^{2r}\frac{(-1)^j}{j!}\psi_j(u,\lambda).
\end{align*}
and
\begin{align*}
&\liminf_{N\to\infty}\frac{c^{-1}(N)}{nh^+(N)}
\left(\e_0\left[\exp\left\{-\sum_{i=1}^m \lambda_i\frac{N}{c^{-1}(N)}L(\tau^-,u_iN)\right\}\right]-1\right)\\
&\hspace{3cm}\geq \sum_{l=1}^{2r+1}\frac{(-1)^j}{j!}\psi_j(u,\lambda)
\end{align*}
respectively. Note that estimate \eqref{phi-bound} allows to let $r\to\infty$ for 
all $\lambda_k$ small enough. As a result, there exists $\delta>0$ such that 
if $\lambda_k\in[0,\delta)$ for all $k$ then
\begin{equation}
\label{asymp-final}
\lim_{N\to\infty}\frac{c^{-1}(N)}{nh^+(N)}
\left(\e_0\left[\exp\left\{-\sum_{i=1}^m \lambda_i\frac{N}{c^-(N)}L(\tau^-,u_iN)\right\}\right]-1\right)
=\Psi(u,\lambda),
\end{equation}
where
$$
\Psi(u,\lambda):=\sum_{j=1}^{\infty}\frac{(-1)^j}{j!}\psi_j(u,\lambda).
$$
Notice also that \eqref{mixed.moments} implies the continuity of $\lambda\mapsto\Psi(u,\lambda)$
on $[0,\delta)^m$.

It is clear that $(L_W(T_M,x_1),L_W(T_M,x_2),\ldots,L_W(T_M,x_m))$ is equal in distribution
to the sum of $M$ independent copies of $(L(\tau^-,x_1),L(\tau^-,x_2),\ldots,L(\tau^-,x_m))$.
Then, in view of \eqref{asymp-final},
$$
\lim_{N\to\infty}\e_0\left[\exp\left\{-\sum_{i=1}^m\lambda_il^{(N)}(u_i)\right\}\right]=e^{-\Psi_(u,\lambda)}
$$
for any sequence $M(N)$ whichis asymptotically equivalent to $nh^+(N)/c^-(N)$.

By the continuity theorem for the Laplace transformation, the distribution of the vector
$(l^{(N)}(u_1),l^{(N)}(u_2),\ldots,l^{(N)}(u_m))$ converges weakly to a law $F_{u}$,
which is characterised by the Laplace transform $\lambda\mapsto e^{-\Psi(u,\lambda)}$.
The continuity of functions $\Psi(u,\lambda)$ implies, in particular, the consistency
of the family of finite dimensional distributions $\{F_u\}$. Thus, the proof is completed.

\end{document}